\documentclass[]{amsart}
\usepackage{amscd,amsthm,amssymb,amsfonts,amsmath,euscript}

\theoremstyle{plain}
\newtheorem{thm}{Theorem}
\newtheorem{lemma}{Lemma}[subsection]

\theoremstyle{definition}

\newtheorem{eg}[lemma]{Example}

\theoremstyle{remark}
\newtheorem{remark}[lemma]{Remark}

\newcommand{\nc}{\newcommand}

\def\makeop#1{\expandafter\def\csname#1\endcsname
  {\mathop{\rm #1}\nolimits}\ignorespaces}
\makeop{Hom}   \makeop{End}   \makeop{Aut}   \makeop{Isom}  \makeop{Pic} 
\makeop{Gal}   \makeop{ord}   \makeop{Char}  \makeop{Div}   \makeop{Lie} 
\makeop{PGL}   \makeop{Corr}  \makeop{PSL}   \makeop{sgn}   \makeop{Spf}
\makeop{Spec}  \makeop{Tr}    \makeop{Nr}    \makeop{Fr}    \makeop{disc}
\makeop{Proj}  \makeop{supp}  \makeop{ker}   \makeop{im}    \makeop{dom}
\makeop{coker} \makeop{Stab}  \makeop{SO}    \makeop{SL}    \makeop{SL}
\makeop{Cl}    \makeop{cond}  \makeop{Br}    \makeop{inv}   \makeop{rank}
\makeop{id}    \makeop{Fil}   \makeop{Frac}  \makeop{GL}    \makeop{SU}
\makeop{Nrd}   \makeop{Sp}    \makeop{Tr}    \makeop{Trd}   \makeop{diag}
\makeop{Res}   \makeop{ind}   \makeop{depth} \makeop{Tr}    \makeop{st}
\makeop{Ad}    \makeop{Int}   \makeop{tr}    \makeop{Sym}   \makeop{can}
\makeop{length}\makeop{SO}    \makeop{torsion} \makeop{GSp} \makeop{Ker}
\makeop{Adm}   \makeop{Mat}
\def\makebb#1{\expandafter\def
  \csname bb#1\endcsname{{\mathbb{#1}}}\ignorespaces}
\def\makebf#1{\expandafter\def\csname bf#1\endcsname{{\bf
      #1}}\ignorespaces} 
\def\makegr#1{\expandafter\def
  \csname gr#1\endcsname{{\mathfrak{#1}}}\ignorespaces}
\def\makescr#1{\expandafter\def
  \csname scr#1\endcsname{{\EuScript{#1}}}\ignorespaces}
\def\makecal#1{\expandafter\def\csname cal#1\endcsname{{\mathcal
      #1}}\ignorespaces} 

\def\doLetters#1{#1A #1B #1C #1D #1E #1F #1G #1H #1I #1J #1K #1L #1M
                 #1N #1O #1P #1Q #1R #1S #1T #1U #1V #1W #1X #1Y #1Z}
\def\doletters#1{#1a #1b #1c #1d #1e #1f #1g #1h #1i #1j #1k #1l #1m
                 #1n #1o #1p #1q #1r #1s #1t #1u #1v #1w #1x #1y #1z}
\doLetters\makebb   \doLetters\makecal  \doLetters\makebf
\doLetters\makescr 
\doletters\makebf   \doLetters\makegr   \doletters\makegr
     \def\qed{\qedmark\medbreak}%
\def\qedmark{{\enspace\vrule height 6pt width 5pt depth 1.5pt}}%

\normalsize

\makeop{Bl}

\def\Fp{{\bbF}_p}

\def\Qp{{\bbQ}_p}

\def\Qbar{\overline{\bbQ}}

\newcommand{\Z}{\mathbb Z}

\newcommand{\Q}{\mathbb Q}

\newcommand{\C}{\mathbb C}

\nc{\embed}{\hookrightarrow}

\newcommand{\ac}{algebraically closed }
\newcommand{\dieu}{Dieudonn\'{e}\,}

\nc{\ol}{\overline}
\nc{\wt}{\widetilde}
\nc{\opp}{\mathrm{opp}}

\makeop{Ram}
\makeop{Rep}

\begin{document}
\renewcommand{\thefootnote}{\fnsymbol{footnote}}
\setcounter{footnote}{-1}
\numberwithin{equation}{subsection}


\title{Abelian varieties without a prescribed Newton Polygon reduction}
\author{Jiangwei Xue and Chia-Fu Yu}
\address{
Institute of Mathematics, Academia Sinica and NCTS (Taipei Office)\\
6th Floor, Astronomy Mathematics Building \\
No. 1, Roosevelt Rd. Sec. 4 \\ 
Taipei, Taiwan, 10617} 
\email{chiafu@math.sinica.edu.tw}
\email{xue\_j@math.sinica.edu.tw}

\date{\today}
\subjclass[2000]{}

\begin{abstract}
  In this article we construct for each integer $g\ge 2$ an abelian
  variety $A$ of dimension $g$ defined over a number field for which
  there exists a symmetric integral slope sequence of length $2g$ that
  does not appear as the slope sequence of $\wt A$ for any good
  reduction $\wt A$ of $A$.
\end{abstract} 

\maketitle


\section{Introduction}
\label{sec:01}

Let $A$ be an abelian variety over a number field $F$. 
It is conjectured that $A$ always has a good ordinary reduction and 
furthermore that 
there is a finite field extension $L$ of $F$ and a density one set
$V(A,L)$ of non-archimedean places of $L$  
such that the base change $A\otimes_F L$ has (good) ordinary reduction
at every $v\in V(A,L)$ (cf. Bogomolov-Zarhin
\cite{bogomolov-zarhin:k3}).   

This conjecture is known to be true for elliptic curves 
(Serre~\cite{serre:abelian}), 
abelian surfaces (Ogus \cite{ogus:900}), 
and some abelian three-folds or four-folds 
(see Noot \cite{noot:compos1995, noot:crelle2000} and 
Tankeev~\cite{tankeev:1999}). 
In \cite{bogomolov-zarhin:k3} Bogomolov and Zarhin prove the analogous
theorem for K3 surfaces.   

Concerning non-ordinary reduction, Elkies \cite{elkies:invent1987}
shows that under a mild condition on the number field $F$, any
elliptic curve over $F$ has good supersingular reduction at infinitely
many primes of $F$. Inspired by work of Elkies and having no
counter-example, one may naturally ask whether any abelian variety
over a number field $F$ admits infinitely many supersingular
reductions. So far this is not known yet even for abelian surfaces
(except for some special cases like CM abelian surfaces).  In the
function field analogue, Poonen \cite{poonen:imrn1998} shows the
existence of a Drinfeld module of rank two which does not admit any
supersingular reduction.

Throughout this paper, $p$ and $\ell$ denote primes in $\Q$. For
abelian varieties of dimension $g$ in positive characteristic, the
attached $p$-divisible groups up to isogeny over an \ac field are
classified by their Newton Polygons, or equivalently, by the
associated slope sequences $\beta$ (the \dieu-Manin theorem, see
Manin~\cite{manin:thesis}). This invariant is a sequence of $2g$
rational numbers
$$ 0\le \lambda_1\le \dots \le \lambda_{2g}\le 1,$$ 
which satisfy the symmetric and integral conditions:
\begin{itemize}
\item [(i)] $\lambda_i+\lambda_{2g+1-i}=1$ for all $1\le i\le 2g$, and
\item [(ii)]  the
multiplicity of each $\lambda_i$ is a multiple of its denominator.
\end{itemize}  
An abelian variety defined over a field of characteristic $p>0$ is
said to be \textit{supersingular} if all $\lambda_i=1/2$; it is said
to be \textit{ordinary} if $\lambda_i$ is either $0$ or $1$ for all
$1\leq i\leq 2g$.  Then given an abelian variety $A$ over $F$ of
dimension $g$ and a symmetric integral slope sequence $\beta$ of
length $2g$, does $A$ always admit a good reduction whose slope
sequence coincides with $\beta$?  In this article we give a negative
answer to this general question.

We will restrict ourselves to the case where $A$ is
absolutely simple. Otherwise, one is reduced to study the simple
factors of $A$ (by extending the base field if necessary). For
example, if $A=E^g$, a $g$-fold product of an elliptic curve $E$, then
the reductions of $A$ are either ordinary or supersingular. In other
words, its reductions miss almost all the symmetric integral slope
sequences except the two ``extremal'' ones.

\begin{thm}\label{12}
  For any integer $g\ge 2$, there is a pair $(A/F, \beta)$
  consisting of
  \begin{itemize}
  \item an absolutely simple abelian variety $A$ of dimension $g$
    defined over a number field $F$,
  \item a symmetric integral slope sequence $\beta$ of length $2g$,
  \end{itemize}
  such that $\beta$ does not occur as the slope sequence of any good
  reduction of $A$.
\end{thm}

In our construction the number field $F$ depends on the dimension $g$.
However, we have the following theorem.

\begin{thm}\label{11}
  In Theorem~\ref{12}, there are infinitely many $g$ for which the
  number field $F$ can be chosen to be $\Q$.
\end{thm}


The CM abelian varieties play an essential role in our construction.
For a CM abelian variety $A/F$ of type $(K, \Phi)$, the Newton polygon
of the reduction of $A$ over a ``good'' prime $\grq\mid p$ of $F$ can
be determined from the CM-type $\Phi$ by the Taniyama-Shimura formula.
This allows us to prove Theorem~\ref{12} by choosing a special type of
CM-abelian varieties. To obtain Theorem~\ref{11}, we study Honda's
examples \cite{honda:osaka1966}. They are the Jacobians of the smooth
projective curves defined by the affine equations $y^2=1-x^\ell$ for
all odd primes $\ell$. This gives a family of CM Jacobians whose
dimensions are of the form $(\ell-1)/2$.


  
\section{Reduction of CM abelian varieties}

First, we recall Tate's formulation of Shimura-Taniyama formula in
terms of $p$-divisible groups~\cite[Section 4]{tate:ht}, 
which describes the behavior of
reductions of CM abelian varieties. With this tool in hand, we then
show that when the CM-field is a cyclic extension of $\Q$, the types
of Newton polygons arising from the reductions are quite limited. This
leads to a proof of Theorem~\ref{12}. At the end, we present Honda's
examples \cite{honda:osaka1966} and gives a proof of Theorem~\ref{11}.

\subsection{Shimura-Taniyama theory for CM abelian varieties}
Let $K$ be a CM-field of degree $2g$ over $\Q$, and $\Sigma_K$ the set
of embeddings of $K$ into the algebraic closure $\Qbar\subset \C$ of
$\Q$ : \[\Sigma_K:=\Hom_\Q(K,\Qbar)=\Hom_{\Q}(K,\C).\] We fix an
embedding $\iota:\Qbar \hookrightarrow \Qbar_p$, where $\Qbar_p$ is a
fixed algebraic closure of $\Qp$. It induces a bijection:
\[ \iota: \Sigma_K \simeq \Sigma_{K,p}:=\Hom_{\Q_p}(K\otimes_{\Q}
\Q_p, \Qbar_p), \qquad \varphi\mapsto \iota\circ\varphi.\]
On the other hand, we have $K\otimes_\Q \Qp=\prod_{\grp | p} K_\grp$,
where $K_\grp$ denotes the completion of $K$ at the prime $\grp$ of
$K$. If we put $\Sigma_{K_\grp}:=\Hom_{\Qp}(K_\grp, \Qbar_p)$, then
\begin{equation}
  \label{eq:1}
 \Sigma_{K,p}=\coprod_{\grp}
\Sigma_{K_\grp}. 
\end{equation}
Let $\rho\in \Sigma_K$ be a fixed embedding of $K$ into $\Qbar$. When
$K/\Q$ is Galois with $G:=\Gal(K/\Q)$, we may identify $\Sigma_K$ (and
in turn $\Sigma_{K,p}$ via $\iota$) with $G$ via $\rho$:
\begin{equation}
  \label{eq:8}
G \simeq \Sigma_K, \quad \sigma \leftrightarrow \rho\circ \sigma,
\;\forall \sigma\in G.
\end{equation}
 The embedding $\iota\circ \rho :
K\hookrightarrow \Qbar_p$ induces a unique prime $\grp_0\mid p$ of
$K$. Let $D_{\grp_0}$ be the decomposition group of $\grp_0$. Then
(\ref{eq:1}) corresponds to the partition of $G$ into the disjoint
union of right cosets of $D_{\grp_0}$. More explicitly,
\begin{equation}
  \label{eq:2}
\Sigma_{K_\grp}\simeq D_{\grp_0}\sigma_\grp^{-1}, \qquad
\text{ with } \quad \sigma_\grp \grp_0=\grp. 
\end{equation}
In particular, $|\Sigma_{K_\grp}|=|D_{\grp_0}|$ for all $\grp \mid p$.
If $K$ is abelian over $\Q$, the decomposition group $D_{\grp_0}$
depends only on $p$ and not on $\grp_0$, so it is denoted by $D_p$
instead. In this case, the partition of $G$ into cosets of $D_p$ does
not depend on the choice of $\iota$ nor $\rho$.

Let $c\in \Aut(K)$ be the unique automorphism of order 2 that is
induced by the complex conjugation for any embedding $K\hookrightarrow
\C$. A subset $\Phi\subset \Sigma_K$ is said to be a CM-type on $K$ if
$\Sigma_K=\Phi \coprod \Phi c$.  Given a CM-type $\Phi$ on $K$, we
write $\Phi_\grp:=\iota(\Phi)\cap \Sigma_{K_\grp}\subset \Sigma_{K,p}$
for prime $\grp$ of $K$. Let $\bar{\grp}:=c\grp$, then $c$ induces an
isomorphism between the completions $K_\grp\simeq K_{\bar{\grp}}$,
which gives rise to a bijective map
\begin{equation}
  \label{eq:3}
  c: \Sigma_{K_{\bar{\grp}}} \to \Sigma_{K_\grp},  \qquad \varphi \mapsto
  \varphi \circ c.
\end{equation}
It follows from the definition of a CM-type that
\begin{equation}
  \label{eq:4}
  \Phi_\grp  \coprod \Phi_{\bar{\grp}}c=\Sigma_{K_\grp}.
\end{equation}
In particular, if $c\grp =\grp$, then
$|\Phi_\grp|=|\Sigma_{K_\grp}|/2$.

A complex abelian variety $A_\C$ of dimension $g$ is said to have
complex multiplication of type $(K,\Phi)$ if there is an embedding
$K\hookrightarrow \End^0(A_\C):=\End(A_\C)\otimes_\Z \Q$, and the
character of the representation of $K$ on the Lie algebra $\Lie_{\C}(A_\C)$ is given by $\sum_{\varphi\in \Phi} \varphi$.  A CM abelian
variety of type $(K, \Phi)$ is simple if and only if $\Phi$ is
\textit{primitive}, i.e., not induced from a CM-type of a proper
CM-subfield of $K$.

Let $A_\C$ be a CM complex abelian variety of type $(K, \Phi)$. Then
$A_\C$ has a model $A$ defined over a number field $F\subset \Qbar$
(\cite[Section 6.2 and 12.4]{shimura-taniyama}). Enlarging the base
field $F$ if necessary, we may assume that $A$ has a good reduction
$A\otimes \kappa(\grq)$ at every finite place $\grq$ of $F$
(cf. \cite{serre-tate}). Here $\kappa(\grq)$ denotes the residue field
of $\grq$. Since we are only concerned with the isogeny invariants,
replacing $A$ with its quotient by a suitable finite subgroup if
necessary, we may further assume that $\End(A)\cap K= O_K$, the ring
of integers of $K$. The $p$-adic completion of $O_K$ decomposes into a
product
\begin{equation}
  \label{eq:5}
  O_K\otimes_\Z \Z_p = \prod_{\grp \mid p } O_{K_\grp}, 
\end{equation}
where $O_{K_\grp}$ denotes the ring of integers in $K_\grp$.  Let
$\grq$ be the prime of $F$ corresponding to the embedding $\iota:
\Qbar\to \Qbar_p$, and
\begin{equation}
  \label{eq:6}
 A\otimes \ol{\kappa(\grq)}[p^\infty]=\prod_{\grp|p} H_\grp   
\end{equation}
be the decomposition of $p$-divisible groups induced from
(\ref{eq:5}).  Then each component $H_\grp$ is of height
$|\Sigma_{K_\grp}|$, dimension $|\Phi_\grp|$, and isoclinic of slope
$|\Phi_\grp|/|\Sigma_{K_\grp}|$ (see \cite[Chapter III, Theorem
1]{shimura-taniyama}, \cite[Section 5]{tate:ht} and \cite[Section
4]{yu:cm}). In particular, if $K$ is abelian over $\Q$, then the slope
sequence of $A\otimes \ol{\kappa(\grq)}$ depends only on $\Phi$ and
$p$. In other words, it is independent of the choice of $\iota:
\Qbar\to \Qbar_p$, and thus independent of the prime $\grq\mid p$ of
$F$ for the reduction.

\begin{eg}\label{eg:exist-ord-ss-reduction}
  Suppose that $p$ splits completely in $K$. Then $K_\grp=\Q_p$ for
  all $\grp \mid p$. So $|\Sigma_{K_\grp}|=1$, and $\Phi_\grp$ either
  coincides with $\Sigma_{K_\grp}$ or is empty. Therefore, each
  $H_\grp$ has slope either $0$ or $1$. The reduction of $A$ at any
  prime  $\grq\mid p$ is ordinary.

  Let $L$ be the Galois closure of $K$ over $\Q$, i.e., the compositum
  of all conjugates of $K/\Q$. It is again a CM-field with $c$ in the
  center of $\Gal(L/\Q)$. Let $p$ be a prime unramified for $L/\Q$
  such that the Artin symbol $(p, L/\Q)=c \in \Gal(L/\Q)$. By
  Tchebotarev density theorem (\cite[Theorem 10, Section
  VIII.4]{lang:ant}), such $p$ exists and they have a positive
  density. Any prime $\grp\mid p$ in $K$ is then fixed by $c$.  By the
  remark below (\ref{eq:4}), the reduction of $A$ at $\grq$ is
  supersingular.
\end{eg}

\subsection{Proof of theorem~\ref{12}}

\begin{lemma}
  For any integer $g\ge 1$, there is a CM field $K$ which is a cyclic
  extension over $\Q$ with Galois group $G\simeq \Z/{2g
    \Z}$. Moreover, $K$ admits a primitive CM-type $\Phi$.
\end{lemma}
\begin{proof}
  By the Dirichlet theorem on arithmetic progressions (cf.~Lang
  \cite[Section VIII.4]{lang:ant}), there is a prime number $\ell$
  such that $\ell\equiv 1+2g \pmod {4g}$. Then the integer
  $m:=(\ell-1)/2g$ is odd. Let $K$ be the fixed subfield of the
  $\ell$-th cyclotomic field $\Q(\zeta_\ell)$ for the unique subgroup
  $H\subset (\Z/\ell\Z)^\times$ of order $m$. Since $|H|$ is odd, the
  complex conjugation $c$ on $\Q(\zeta_\ell)$ is not contained in $H$,
  and hence it induces a non-trivial automorphism of $K$. Therefore,
  $K$ is a CM field which is cyclic over $\Q$ of degree $2g$.  We claim
that $\Phi=\{0, 1, \cdots, g-1\}\subseteq \Z/2g\Z=\Gal(K/\Q)$ is a
primitive CM-type on $K$. Otherwise, $\Phi$ will be translation
invariant under a nontrivial subgroup $\Gal(K/K')\subset \Z/2g\Z$ for
some proper CM-subfield $K'$ of $K$, but this is not the case. \qed
\end{proof}

Now for any $g\geq 2$, let $K\subset \Qbar$ be a CM-field cyclic over
$\Q$ with Galois group $G=\Z/2g\Z$. Choose a primitive CM type
$\Phi\subset \Sigma_K\simeq G$. The complex torus $\C^\Phi/\Phi(O_K)$
defines a complex abelian variety $A_\C$ of CM type $(K,\Phi)$. Let
$A$ be a model of $A_\C$ defined over a sufficiently large number
field $F\subset \C$ such that $A$ has good reduction everywhere.  Let
$\grq$ be a prime of $F$ over $p$, and $D_p$ be decomposition group of
$p$ in $K$. Since $\Gal(K/\Q)$ is cyclic, $D_p$ is uniquely determined
by its order $f:=|D_p|$. We claim that the slope sequence of the
reduction $A\otimes \kappa(\grq)$ is uniquely determined by
$f$. Indeed, the slope of each component $H_\grp$ in (\ref{eq:6}) is
of the form $\lambda:=|\Phi\cap (a+D_p)|/f$ for some coset $a+D_p$ of
$D_p$ in $G$.  If $f$ is even, the complex conjugation $c\in
\Gal(K/\Q)$ lies in $D_p$, hence every prime $\grp\mid p$ in $K$ is
fixed by $c$.  It follows that every $H_\grp$ in (\ref{eq:6}) is
isoclinic of slope $1/2$, and thus $A\otimes {\kappa(\grq)}$ is
supersingular. For a non-supersingular reduction of $A$, $f$ is
necessarily odd, so all the slopes $\lambda$ have odd denominators.

Let $M_g$ be the number of all possible slope sequences arising from
good reductions of $A$, and $N_g$ be the number of all symmetric
integral slope sequences of length $2g$. To prove Theorem~\ref{12}, it
is enough to show that $M_g<N_g$. By the previous arguments, we have
an upper bound
\begin{equation}
  \label{eq:7}
 M_g \le 1+ \text{ the number of positive odd divisors of
  $2g=|G|$}. 
\end{equation}
Therefore, $M_g\leq g$ for any $g\geq 2$.  On the other hand, for any
$g\geq 1$, one easily sees $N_g\ge g+1$ by counting the number of
symmetric integral slope sequences taking values only in $\{0, 1/2,
1\}$. This shows that $M_g< N_g$ for all $g\ge 2$ and
hence proves Theorem~\ref{12}.

In fact, let $\beta$ be a symmetric integral slope sequence of length
$2g$ that takes values only in $\{0, 1/2, 1\}$, and $\beta$ is neither
ordinary nor supersingular.  There are $g-1$ such slope sequences. We
have shown that $\beta$ never coincides with the slope sequence of any
good reduction of $A$.

\begin{remark}
  Suppose that $g=2^n$. By (\ref{eq:7}), $M_g\leq 2$, and hence it is
  $2$ by Example~\ref{eg:exist-ord-ss-reduction}. Varying $n$, we
  obtain an infinite family of absolutely simple abelian varieties
  whose reductions are either ordinary or supersingular. There are
  other class of abelian varieties that enjoy this property.  For
  example, let $D$ be an indefinite quaternion division algebra over
  $\Q$, and $A/F$ be an abelian surface over a number field $F$ with
  quaternion multiplication (QM) by $D$. In other words, there exists
  an embedding $D\hookrightarrow \End^0(A)$.  Then any reduction of
  $A$ is either ordinary or supersingular. Indeed, let $\wt A$ be an
  good reduction of $A$ over some prime $\grq\mid p$ of $F$. The
  quaternion algebra $D_p:=D\otimes_\Q \Q_p$ acts on $V_p{\wt
    A}:=T_p{\wt A}\otimes_{\Z_p}\Q_p$, where $T_p{\wt A}$ is the
  Tate-module of ${\wt A}$. If ${\wt A}$ has slope sequence $(0, 1/2,
  1/2, 1)$, then $V_p{\wt A}$ is a dimension one $\Q_p$-vectors space,
  which can not admit any action by $D_p$. It is known that an abelian
  surface $A$ with QM by $D$ is absolutely simple if and only if it
  does not have CM (cf. \cite{yu:qm}). If $A$ has no CM, then
  $\End^0(A)=D$, otherwise, $A$ is isogenous to the self-product of a
  CM elliptic curve. In particular, any good reduction $\wt A$ of $A$,
  which is an abelian surface with QM by $D$ over a finie field, 
  is always isogenous to the self-product of an elliptic curve.
  \end{remark}

We would like to thank the referee for the following stronger version
of Theorem~\ref{12}. 
\begin{thm}\label{thm:main-res-improved}
  Let $K$ be a CM-field with $K/\Q$ Galois and $[K:\Q]=2g$ with $g\geq
  2$. Let $A/F$ be an abelian variety with CM by $K$. There exists a
  symmetric integral slope sequence $\beta$ such that, for each prime
  $\grq$ of $F$, the slope sequence of reduction $A\otimes
  \kappa(\grq)$ does not coincide with $\beta$.
\end{thm}

\begin{proof}
  Since $K$ is Galois over $\Q$, in the decomposition (\ref{eq:6}) of
  the $p$-divisible group of the reduction, each component $H_\grp$ is
  isoclinic of slope $|\Phi_\grp|/f$, where $f$ is the order of the
  decomposition group at $p$ in $K$. Suppose that $g\geq 4$, we pick
  $1<f_0<g$ such that $\gcd(f_0, 2g)=1$. For example, if $g$ is odd,
  we may choose $f_0=g-2$, and if $g$ is even, we choose
  $f_0=g-1$. Then $\Gal(K/\Q)$ has no subgroup of order $f_0$, and
  thus $f_0$ can never occur as the denominator of a slope of some
  $H_\grp$. Let $\lambda=1/f_0$, and $\beta$ be the following symmetric
  integral slope sequence:
\[\beta= (\underbrace{0, \cdots, 0}_{g-f_0} \quad
\underbrace{\lambda, \cdots, \lambda}_{f_0},\quad
\underbrace{1-\lambda, \cdots, 1-\lambda}_{f_0}, \quad \underbrace{1,
  \cdots, 1}_{g-f_0}). \] Then $\beta$ never occurs as the slope
sequence of any good reduction of $A$.

If $g=3$, then $\Gal(K/\Q)$ is a group of order $6$ with an element of
order $2$ in its center. Therefore, $\Gal(K/\Q)=\Z/6\Z$ and we are
reduced to the proof of Theorem~\ref{12}. If $g=2$, then $\Gal(K/\Q)$
is either $\Z/4\Z$ or $(\Z/2\Z)^2$.  In the first case, once again we
are reduced to the proof of Theorem~\ref{12}. In the second case $K$
contains two quadratic imaginary subfields, any CM-type on $K$ is
induced from one of them. Hence $A$ is isogenous over $\bar{F}$ to a
product of CM elliptic curves, so its reduction is either ordinary or
supersingular. It never achieves the slope sequence $(0, 1/2, 1/2, 1)$
from its reductions.
\end{proof}

\subsection{Honda's examples} We will describe some results of Honda
\cite{honda:osaka1966} and prove Theorem~\ref{11}.

Let $\ell$ be an odd prime, and $C=C_\ell$ be the smooth projective
curve over $\Q$ defined by the affine equation
\begin{equation}\label{eq:21}
  y^2=1-x^\ell.
\end{equation}
The genus $g:=g(C)$ of $C$ is $(\ell-1)/2$. The curve $C$ and its
Jacobian $J=J_\ell$ have good reduction at all primes $p\neq
2,\ell$. For a fixed odd prime $p\neq \ell$, let $\wt J$ (resp. $\wt
C$) be the reduction of $J$ (resp. $C$) at $p$ (over $\Fp$).

Let $\zeta_\ell$ be a primitive $\ell$-th root of unity in
$\Qbar\subset \C$, and $K:=\Q(\zeta_\ell)$ be the $\ell$-th cyclotomic
field. Then $K$ is a CM-field that's cyclic over $\Q$ with Galois
group $G:=\Gal(K/\Q)=(\Z/\ell\Z)^\times$, where each $a\in
(\Z/\ell\Z)^\times$ corresponds to the automorphism of $K$ that send
$\zeta_\ell$ to $\zeta_\ell^a$. Let $\rho: K\hookrightarrow \Qbar$ be
the natural inclusion. We will identify $G$ with $\Sigma_K$ via $\rho$
as in (\ref{eq:8}).

The automorphism of $C \otimes_\Q K$ defined by $(x,y)\mapsto (x,
\zeta_\ell y)$ induces an embedding
$K\hookrightarrow \End^0_{K}(J)$. This realizes $J\otimes_\Q K$ as a
CM-abelian variety of type $(K, \Phi)$, where $\Phi=\{1, 2, \cdots
g-1, g\}\subset G$ is a primitive CM-type on $K$.  Let $f$ be the
order of $p$ in $(\Z/\ell\Z)^\times$. The Artin symbol $(p, K/\Q)$
equals to $p\in (\Z/\ell\Z)^\times$, and $f$ is the order of the
decomposition group $D_p=\langle p \rangle \subseteq
(\Z/\ell\Z)^\times$ of $p$ in $K$. By \cite[Theorem 1,
2]{honda:osaka1966} or the proof of Theorem~\ref{12}, the slope
sequence of $\wt J$ depends only on $f$. Moreover, if $f$ is even,
$\wt J$ is supersingular.

Now Theorem~\ref{11} follows from Theorem~\ref{12} by noting that $J$
is an absolutely simple CM abelian variety of dimension $(\ell-1)/2$
defined over $\Q$, and the field $K$ is cyclic over $\Q$.

\section*{Acknowledgments}
We would like to thank the referee for his/her suggestion on the
exposition of the paper. (S)He also made several insightful remarks about
the results, and provided a stronger version of our main result, which
becomes Theorem~\ref{thm:main-res-improved} of the current paper. The
first named author is partially supported by the grant NSC
102-2811-M-001-090.  The second named author is partially supported by
the grants NSC 100-2628-M-001-006-MY4 and AS-98-CDA-M01.



\end{document}